\documentclass[10pt, reqno]{amsart}

\usepackage[utf8]{inputenc}
\usepackage{amsmath}
\usepackage{amssymb}
\usepackage{mathrsfs}
\usepackage{graphics}
\usepackage[dvipdf]{graphicx}
\usepackage{geometry}
\usepackage{setspace}
\usepackage{color}
\usepackage{amsthm}
\usepackage{paralist}
\usepackage{dsfont}
\usepackage{verbatim}
\usepackage[numbers]{natbib}

\newtheorem{defn}{Definition}[section]

\newtheorem{prop}[defn]{Proposition}
\newtheorem{thrm}[defn]{Theorem}
\newtheorem{lem}[defn]{Lemma}
\newtheorem{ass}[defn]{Assumption}
\newtheorem{rem}[defn]{Remark}

\numberwithin{equation}{section}

\newcommand{\grad}{\nabla}
\newcommand{\Div}{\nabla \cdot}
\DeclareMathOperator{\tr}{tr}
\newcommand{\ceil}[2]{\lceil {#1}/{#2} \rceil}

\newcommand{\R}{\mathbb{R}}
\newcommand{\Leb}{\mathcal{L}}
\newcommand{\N}{\mathbb{N}}

\newcommand{\Prob}{{\mathscr{P}}(\Omega)}

\newcommand{\Time}{\infty}
\newcommand{\utn}{u_{\tau}^n}

\newcommand{\utnn}{u_{\tau}^{n-1}}
\newcommand{\ut}{u_\tau}

\newcommand{\IOx}[1]{\int_\Omega #1 \,\mathrm{d}x}

\newcommand{\IOxtime}[1]{\int_0^\infty \int_\Omega #1 \,\mathrm{d}x\,\mathrm{d}t}

\newcommand{\dS}{\,\mathrm{d}\sigma}

\newcommand{\W}{\mathbf{W}_2}

\newcommand{\flowS}{\mathsf{S}}
\renewcommand{\phi}{\varphi}
\newcommand{\eps}{\varepsilon}

\newcommand{\Ftilde}{\widetilde{F}}

\newcommand{\Non}{\mathcal{N}}

\newcommand{\dff}{\mathrm{D}}

\newcommand{\Id}{\mathds{1}}

\newcommand{\trc}{\operatorname{tr}}
\newcommand{\dd}{\,\mathrm{d}}
\newcommand{\auxil}{\mathfrak{A}}
\newcommand{\ban}{\mathrm{X}}
\newcommand{\dg}{\mathbf{g}}
\newcommand{\ent}{\mathcal{E}}
\newcommand{\pot}{\mathcal{V}}
\newcommand{\ce}{\mathsf{e}}
\newcommand{\id}{\operatorname{id}}
\newcommand{\tT}{\mathrm{T}}

\renewcommand{\tilde}{\widetilde}
\renewcommand{\bar}{\overline}

\begin{document}

\begin{abstract}
We prove the global-in-time existence of nonnegative weak solutions to a class of fourth order partial differential equations on a convex bounded domain in arbitrary spatial dimensions. Our proof relies on the formal gradient flow structure of the equation with respect to the $L^2$-Wasserstein distance on the space of probability measures. We construct a weak solution by approximation via the time-discrete minimizing movement scheme; necessary compactness estimates are derived by entropy-dissipation methods. Our theory essentially comprises the thin film and Derrida-Lebowitz-Speer-Spohn equations.
\end{abstract}

\title[Existence for fourth order equations with Wasserstein gradient structure]{Existence of weak solutions to a class of fourth order partial differential equations with Wasserstein gradient structure}
\author[Daniel Loibl]{Daniel Loibl}
\address{Zentrum f\"ur Mathematik \\ Technische Universit\"at M\"unchen \\ 85747 Garching, Germany}
\email{daniel.loibl@gmail.com}
\author[Daniel Matthes]{Daniel Matthes}
\address{Zentrum f\"ur Mathematik \\ Technische Universit\"at M\"unchen \\ 85747 Garching, Germany}
\email{matthes@ma.tum.de}
\author[Jonathan Zinsl]{Jonathan Zinsl}
\address{Zentrum f\"ur Mathematik \\ Technische Universit\"at M\"unchen \\ 85747 Garching, Germany}
\email{zinsl@ma.tum.de}
\keywords{Fourth-order equations; gradient flow; Wasserstein distance; weak solution; minimizing movement scheme}
\date{\today}
\subjclass[2010]{Primary: 35K30; Secondary: 35A15, 35D30.}

\maketitle

\section{Introduction}\label{sec:intro}

\subsection{Main results}\label{subsec:main}

This work is concerned with nonnegative solutions $u:[0,\infty)\times\Omega\to \R_+$ to the partial differential equation
\begin{align}
\partial_t u(t,x)&=\Div u(t,x) \grad \left[F_z(x, u(t,x), \grad u(t,x)) - \Div F_p(x, u(t,x),\grad u(t,x))\right]\label{PDE00}
\end{align}
for $(t,x)\in(0,\Time) \times \Omega$, where $\Omega\subset\R^d$ $(d\ge 1)$ is a bounded and convex domain with smooth boundary $\partial\Omega$ and exterior unit normal vector field $\nu$. Additionally, the sought-for solution $u$ is subject to the no-flux and homogeneous Neumann boundary conditions
\begin{align}
u(t,x)\partial_\nu \left[F_z(x, u(t,x), \grad u(t,x)) - \Div F_p(x, u(t,x),\grad u(t,x))\right]  = 0 &= \partial_\nu u(t,x)  \label{PDEb2}
\end{align}
for $t>0$ and $x\in \partial\Omega$, and to the initial condition 
\begin{align}
u(0,\cdot) = u_0\in L^1(\Omega), ~u_0 \geq 0 \text{ and }  \IOx{ u_0(x)} =1.\label{eq:ic}
\end{align}
For the nonlinearity $F$ -- to which we refer as \emph{Lagrangian} -- we assume
\begin{ass}[Conditions for $F$]
\label{ass}
$F\in C^2(\overline\Omega \times \R_{+} \times \R^d)$, with:
\begin{enumerate}[(i)]	
	\item  \label{ass1}
		{\emph{Radial symmetry}: There is some $\Ftilde:\overline{\Omega}\times \R_+\times\R_+\to\R$ with $\partial_r\Ftilde(x,z,r)\ge 0$ such that
			$F(x,z,p) = \Ftilde(x,z,|p|) \text{ for all $p\in\R^d$}$.} 			
	\item \label{ass2}
		{\emph{Convexity}: There exists $\gamma>0$ such that
\begin{align*}
\dff_{(x,z,p)}^2 F(x,z,p)[(\xi,\zeta,\pi),(\xi,\zeta,\pi)]&\ge \gamma | \pi |^2
\end{align*} $  \text{ for all  } (\xi,\zeta,\pi) \in \R^d \times \R \times \R^d.$ and all $(x,z,p)\in \overline\Omega \times \R_{+} \times \R^d$.}
	\item  \label{ass3}
		{\emph{Bounds}: There are constants $C\ge c >0$ such that
		$$c|p|^2 \le F(x,z,p) \le C(|p|^2+1)$$
		 for every $x \in \overline\Omega$, $z \in \R_{+}$, $p \in \R^d$.
		Furthermore, there exists $D>0$ such that
		$$|F_x(x,z,p)|, ~\tr(F_{xx}(x,z,p)),~ z |F_z(x,z,p)|,~ |F_p(x,z,p)|^2 \le D(|p|^2+1).$$}
\end{enumerate}
\end{ass}
Our main result is concerned with the existence of solutions to \eqref{PDE00}--\eqref{eq:ic}:
\begin{thrm}[Existence of weak solutions]\label{thm:existF}
Assume that $F$ satisfies Assumption \ref{ass} and let a nonnegative initial datum $u_0\in H^1(\Omega)$ with $\|u_0\|_{L^1}=1$ be given. Then, there exists a nonnegative \emph{weak solution} $u:[0,\infty)\times\Omega\to[0,\infty]$ to \eqref{PDE00}--\eqref{eq:ic} in the following sense: For all $\phi \in C^\infty(\overline\Omega)$ with $\partial_\nu \phi = 0$ on $\partial \Omega$ and all $\eta \in C^{\infty}_c((0,\Time))$ with $\eta(0)=0$, one has
\begin{align}
\label{weakform}
\IOxtime{ u(t,x) \phi(x) \partial_t \eta(t) } = \IOxtime{ \Non(x,u(t,x),\phi(x))\eta(t)} ,
\end{align}
where
$$\Non(x,\rho,\psi) := F_x(x,\rho,\grad\rho) \cdot \grad \psi + (F (x,\rho,\grad\rho)- \rho F_z(x,\rho,\grad\rho))\Delta \psi + F_p(x,\rho,\grad\rho) \cdot (\grad \rho \Delta \psi + \rho \grad\Delta\psi +  \grad^2\psi \grad \rho).$$
Particularly, $u\in L^\infty([0,T];H^1(\Omega))\cap L^2([0,T];H^2(\Omega))$ for each $T>0$, and $u(t,\cdot)$ is a probability density on $\Omega$ at each $t\ge 0$. Furthermore, the map $t\mapsto \int_\Omega F(x,u(t,x),\nabla u(t,x))\dd x$ is almost everywhere equal to a non-decreasing function.
\end{thrm}
The second part of this paper is devoted to a special case of equation \eqref{PDE00} where the Lagrangian $F$ does \emph{not} satisfy the conditions from Assumption \ref{ass}: Specifically, we let 
\begin{align}F(x,z,p):=\frac12 f'(z)^2 |p|^2,\label{eq:speziell}\end{align}
for a map $f:\R_+\to\R$ subject to
\begin{ass}[Conditions for $f$]\label{assf}~
\begin{enumerate}[(i)]
\item $f\in C^3((0,\infty))$, $f(0)=0$, $f''(z)<0$ for all $z>0$ and there exist $C>0$ and $\alpha\in \left(\left[\frac12-\frac1{d}\right]_+,1\right)$ such that for all $z>0$: $f'(z)\ge Cz^{\alpha-1}$. 
\item The limit $\lim\limits_{z\searrow 0}zf'(z)$ exists.
\item There exists $\bar\delta>0$ such that $\frac{f'''(z)f'(z)}{f''(z)^2}\ge \bar\delta+1-\frac{d}{2}+\frac12 \sqrt{d^2+8d}$ for all $z>0$.
\end{enumerate}
\end{ass}
A typical example satisfying Assumption \ref{assf} is the square root $f(z)=\sqrt{z}$ --- see also Remark \ref{rem:power} below. Note that since $f'$ is assumed to be unbounded at zero, especially the convexity assumption \eqref{ass2} is not met by $F$. Nevertheless, we have
\begin{thrm}[Existence of weak solutions: non-convex case]\label{thm:existf}
Assume that $F$ is as in \eqref{eq:speziell}, and that $f$ satisfies Assumption \ref{assf}. Let moreover a nonnegative initial datum $u_0$ with $\|u_0\|_{L^1}=1$ and $f(u_0)\in H^1(\Omega)$ be given. Then, there exists a nonnegative \emph{weak solution} $u:[0,\infty)\times\Omega\to[0,\infty]$ to \eqref{PDE00}--\eqref{eq:ic}: For all $\phi \in C^\infty(\overline\Omega)$ with $\partial_\nu \phi = 0$ on $\partial \Omega$ and all $\eta \in C^{\infty}_c((0,\Time))$ with $\eta(0)=0$, one has
\begin{align*}
\IOxtime{ u(t,x) \phi(x) \partial_t \eta(t) } = \IOxtime{ \Non_f(u(t,x),\phi(x))\eta(t)} ,
\end{align*}
where 
$$\Non_f(\rho,\psi):=\Delta f(\rho)\nabla f(\rho)\cdot \nabla \psi+\rho f'(\rho)\Delta f(\rho)\Delta\psi.$$
Particularly, one has $u\in L^\infty([0,T];L^p(\Omega))$ for some $p>1$ and $f(u)\in L^\infty([0,T];H^1(\Omega))\cap L^2([0,T];H^2(\Omega))$ for each $T>0$; and $u(t,\cdot)$ is a probability density on $\Omega$ at each $t\ge 0$. Furthermore, the map $t\mapsto \int_\Omega \frac12|\nabla f(u)|^2\dd x$ is almost everywhere equal to a non-decreasing function.
\end{thrm}
\subsection{Strategy of proof}\label{subsec:strat}
The cornerstone of the proofs of Theorems \ref{thm:existF} and \ref{thm:existf} is the variational structure of equation \eqref{PDE00}: Formally, it can be written as a gradient flow of the free energy functional $\Phi : \Prob \rightarrow \R \cup \{+\infty\}$ defined via
\begin{align}\label{eq:energy}
&\Phi(\mu) := \begin{cases} \IOx{F(x, u(x), \grad u(x))}, &\text{ if } \mu=u \cdot \Leb^d \text{ and } u \in H^1(\Omega),\\
+\infty,& \text{ otherwise}.\end{cases}
\end{align}
with respect to the $L^2$-Wasserstein distance $\W$ on the space of probability measures $\Prob$ on $\Omega$. Indeed, for sufficiently regular $F$ and $u$, one finds
\begin{align*}
\partial_t u&=\Div u\grad\frac{\delta \Phi(u)}{\delta u},
\end{align*}
where $\frac{\delta \Phi(u)}{\delta u}$ is the first variation of $\Phi$ in $L^2$ at $u$. Since $\Phi$ is finite on a subspace of absolutely continuous probability measures only, we shall from now on (by a slight abuse of notation) identify the measure $\mu$ with its Lebesgue density $u$ and also write $u\in \Prob$.

Note that the energy functional $\Phi$ is in general \emph{not} $\lambda$-convex along geodesics in $(\Prob,\W)$ \cite{carrillo2009}. Hence, the theory developed by Ambrosio, Gigli and Savaré \cite{savare2008} is not immediately applicable. Still, a continuous flow for \eqref{PDE00} can be defined as the time-continuous limit of the so-called \emph{minimizing movement scheme}, dating back to De Giorgi \cite{degiorgi1983}: 

Fix a step size $\tau>0$ and define a sequence $(u_\tau^n)_{n\in\N}$ recursively by
\begin{align}
u_\tau^0:=u_0,\quad u_\tau^{n}\in\underset{u\in \Prob}{\operatorname{argmin}}\left(\frac1{2\tau}\W^2(u,u_\tau^{n-1})+\Phi(u)\right) \quad(n\in\N).\label{eq:minmov}
\end{align}
With this sequence, we define the \emph{discrete solution} $u_\tau$ as the piecewise constant interpolant, i.e. for each $t\ge 0$, we set
\begin{align}
u_\tau(t,\cdot)=u_\tau^n\quad\text{for }n=\ceil{t}{\tau}.\label{eq:discsol}
\end{align}

The strategy of proof for Theorems \ref{thm:existF} and \ref{thm:existf} can now be summarized as follows: First, we prove that the scheme \eqref{eq:minmov}\&\eqref{eq:discsol} is well-defined, that is, the successive minimizers in \eqref{eq:minmov} exist. By an application of the flow interchange technique from \cite{matthes2009}, we then derive an additional regularity property for the time-discrete solution $u_\tau$. In the framework of Assumption \ref{ass}, the respective entropy-dissipation estimate formally amounts to
\begin{align}
-\frac{\dd}{\dd t}\int_\Omega u\log(u) \dd x&\ge \gamma\int_\Omega \|\grad^2 u\|^2\dd x-C.\label{eq:formal}
\end{align}
We consequently are in position to derive an approximate weak formulation satisfied by the time-discrete solution curves, and can pass to the continuous-time limit $\tau\searrow 0$ afterwards recovering the weak formulation of the original problem \eqref{weakform}.

\subsection{Examples and relation to the literature}\label{subsec:exlit}
The method used here is by now almost a standard technique to construct weak solutions to equations with a formal gradient flow structure. Since the seminal paper by Jordan, Kinderlehrer and Otto \cite{jko1998} on the variational structure of the Fokker-Planck equation, various second order equations (e.g. \cite{otto2001}), fourth order equations (e.g. \cite{giacomelli2001, gianazza2009, matthes2009, lisini2012}) and systems (e.g. \cite{laurencot2011, carrillo2012, blanchet2014, zinsl2012}) have been analysed using this strategy. As models for physical or biological processes, sensible solutions should admit nonnegative values only. Since there exists no general comparison principle for equations of fourth order, this is a nontrivial issue. Thus, considering the dynamics on a suitable space of nonnegative measures, nonnegativity can be directly obtained from existence --- and a formal Wasserstein gradient flow structure might be an indication for it. However, in more specialized situations, other methods could be used (see e.g. \cite{bleher1994, bertozzi1996, bertsch1998, dalpasso1998, juengel2000}).\\
Our framework comprises two important examples:
\begin{itemize}
\item The \emph{thin film equation}
\begin{align}
\partial_t u&=-\nabla\cdot u\nabla\Delta u,\label{eq:thinfilm}
\end{align}
modelling the spreading of a liquid film over a solid surface, can be interpreted as the evolution of the Wasserstein gradient flow of the \emph{Dirichlet energy} $\Phi(u)=\frac12 \|\nabla u\|_{L^2}^2$ \cite{matthes2009}, which fits into the framework of Assumption \ref{ass} for $F(x,z,p)=\frac12|p|^2$. Equation \eqref{eq:thinfilm} can also be considered in the more general form
\begin{align}
\partial_t u&=-\nabla\cdot m(u)\nabla\Delta u,\label{eq:thinfilmmob}
\end{align}
with a possibly nonlinear \emph{mobility function} $m:\R_+\to\R_+$. For the physically relevant cases, the existence of solutions to \eqref{eq:thinfilmmob} has been shown e.g. in \cite{bertsch1998, dalpasso1998}. Note that -- with respect to a modified version of the Wasserstein distance -- a formal gradient structure of \eqref{eq:thinfilmmob} may also be available \cite{lisini2012}. In the simplified framework of one spatial dimension, the long-time asymptotics of the classical thin film equation \eqref{eq:thinfilm} have been analysed in \cite{carrillo2002}, also with entropy-dissipation methods.
\item The \emph{Derrida-Lebowitz-Speer-Spohn equation} 
\begin{align}
\partial_t u=\nabla\cdot u\nabla\left(\frac{\Delta\sqrt{u}}{\sqrt{u}}\right),\label{eq:dlss}
\end{align}
arises in the description of interface fluctuations in the Toom model \cite{dlss1,dlss2} as well as in models for semiconductors (see e.g. \cite{juengel2008}). Often written in the form of the \emph{quantum drift diffusion equation}
\begin{align*}
\partial_t u=\nabla\cdot u\nabla\left(\frac{\Delta\sqrt{u}}{\sqrt{u}}+V\right),
\end{align*}
with a \emph{confinement potential} $V:\R^d\to\R$, existence and long-time behaviour of solutions were studied e.g. in \cite{bleher1994, juengel2000, juengel2008, gianazza2009, matthes2009}. In \cite{gianazza2009, matthes2009}, the  formal Wasserstein gradient flow structure of \eqref{eq:dlss} with respect to the \emph{Fisher information functional} $\Phi(u)=4\|\nabla \sqrt{u}\|_{L^2}^2$ is employed to construct weak solutions via the minimizing movement scheme. 
\end{itemize} 
Actually, one may consider the more general class of functionals $\Phi_\beta(u)=\int_\Omega u^\beta |\nabla u|^2\dd x$ for $\beta\le 0$, which are in general \emph{not} geodesically $\lambda$-convex with respect to the $L^2$-Wasserstein distance. Note, however, that in the range $\beta\in[-5,-4]$, displacement convexity can be proved, see \cite{carrillo2009}. Theorem \ref{thm:existf} presented here especially comprises the range $\beta\in [-1,0)$, by considering $f(z)=z^\alpha$ with $\alpha=1+\frac{\beta}{2}$, which is in accordance to Assumption \ref{assf}. Thus, this work provides a generalization of the existence results from \cite{matthes2009} for a wider class of energy functionals.

\subsection{Outline of the paper}\label{subsec:plan}
After a brief preliminary section, we prove the well-posedness of the minimizing movement scheme under Assumption \ref{ass} in Section \ref{sec:minmov}. After that, the additional regularity of the discrete solution is deduced in Section \ref{sec:addreg}, and a discrete weak formulation is derived in Section \ref{sec:discreteweak}. The proof of Theorem \ref{thm:existF} is then completed by passing to the continuous-time limit (Section \ref{sec:limit}). The extension of the strategy to the framwork of Assumption \ref{assf} will be sketched in Section \ref{sec:f}.


\section{Preliminaries}
In this section, we briefly summarize basic facts about gradient flows in the space of probability measures on the (bounded) set $\Omega$. For more details on optimal transport and gradient flows, we refer to the monographs by Villani \cite{villani2003} and Ambrosio \emph{et al.} \cite{savare2008}.

A sequence $(\mu_n)_{n\in\N}$ in $\Prob$ is said to \emph{converge narrowly} to some limit probability measure $\mu\in\Prob$ if for all continuous and bounded maps $f:\Omega\to\R$, one has
\begin{align*}
\lim_{n\to\infty}\int_\Omega f(x)\dd\mu_n(x)&=\int_\Omega f(x)\dd \mu(x).
\end{align*}
Note that if all measures are absolutely continuous and their respective densities converge weakly in $L^p(\Omega)$ for some $p\ge 1$, narrow convergence of $\mu_n$ to $\mu$ follows.

The space $\Prob$ can be endowed with the so-called \emph{$L^2$-Wasserstein distance} $\W$ defined as
\begin{align*}
\W(\mu_0,\mu_1) := \left( \inf_{\gamma \in \Gamma(\mu_0,\mu_1)} \int_{\Omega\times\Omega} |x-y|^2\, \mathrm{d\gamma}(x,y) \right)^{1/2},
\end{align*}
where $\Gamma(\mu_0,\mu_1)$ denotes the set of all transport plans from $\mu_0$ to $\mu_1$, i.e.
$$\Gamma(\mu_0,\mu_1) := \left\{ \gamma \in \mathscr{P}(\Omega \times \Omega) : \gamma \text{ has marginals } \mu_0 \text{ and } \mu_1 \right\}.$$ 
Since $\Omega$ is a bounded subset of $\R^d$, the infimum above is attained without further restrictions (e.g. finiteness of second moments) on the measures $\mu_0$ and $\mu_1$. Moreover, narrow convergence and convergence with respect to $\W$ are equivalent. Furthermore, if $\mu_0$ is absolutely continuous with respect to the Lebesgue measure, the minimizer for $\W$ is given in terms of the push-forward $\mu_1=T_\#\mu_0$ for a measurable map $T:\Omega\to\Omega$, and the Wasserstein distance reads
\begin{align}
\label{Wasser}
\W(\mu_0,\mu_1)^2 = \int_\Omega |x - T(x)|^2 \,\mathrm{d}\mu_0(x).
\end{align}
In this non-Euclidean framework, one can define a notion of convexity via the so-called \emph{displacement interpolation}: A functional $\auxil:\Prob\to\R\cup\{+\infty\}$ is called \emph{displacement $\lambda$-convex} for some $\lambda\in\R$ if for each $\mu_0,\mu_1\in\Prob$, there exists a constant-speed geodesic curve $\mu_s:[0,1]\to\Prob$ connecting $\mu_0$ and $\mu_1$ such that
\begin{align*}
\auxil(\mu_s)\le (1-s)\auxil(\mu_0)+s\auxil(\mu_1)-\frac{\lambda}{2}s(1-s)\W(\mu_0,\mu_1)^2.
\end{align*}
For displacement convex functionals, we shall use the following notion of \emph{gradient flow}:
\begin{defn}[$\kappa$-flows \cite{savare2008}]\label{thm:0flow}
Let $\auxil:\Prob\to\R\cup\{+\infty\}$ be proper, lower semicontinuous and $\kappa$-displacement convex w.r.t. $\W$ for some $\kappa\in\R$. A continuous semigroup $\flowS^{\auxil}$ on $(\Prob,\W)$ satisfying
the \emph{evolution variational estimate}
 \begin{align*}
   \frac{1}{2}\frac{\dd^+}{\dd s}\W^2(\flowS_{s}^{\auxil}(w),\tilde w)+\frac{\kappa}{2}\W^2(\flowS_{s}^{\auxil}(w),\tilde w)+\auxil(\flowS_{s}^{\auxil}(w))
   &\le\auxil(\tilde w)
\end{align*}
for arbitrary $w,\tilde w$ in the domain of $\auxil$ and for all $s\ge 0$, as well as the monotonicity condition
\begin{align*}
\auxil(\flowS_t^\auxil(w))\le \auxil(\flowS_s^\auxil(w))\quad\forall 0\le s\le t
\end{align*}
for all $w\in\Prob$,
is called \emph{$\kappa$-flow} or \emph{gradient flow} of $\auxil$.
\end{defn}
The cornerstone of the rigorous derivation of \eqref{eq:formal} is
\begin{thrm}[Flow interchange lemma {\cite[Thm. 3.2]{matthes2009}}]\label{thm:flowinterchange}
Let $\auxil$ be a proper, lower semicontinuous and displacement $\lambda$-convex functional on $(\Prob,\W)$ and assume that there exists a $\lambda$-flow $\flowS^\auxil$. Let furthermore $\Psi$ be another proper, lower semicontinuous functional on $(\Prob,\W)$ such that $\operatorname{Dom}(\Psi)\subset \operatorname{Dom}(\auxil)$. Assume that, for arbitrary $\tau>0$ and $\tilde w\in \Prob$, the functional $\frac1{2\tau}\W(\cdot,\tilde w)^2+\Psi$ possesses a minimizer $w$ on $\Prob$. Then, the following holds:
\begin{align*}
\auxil(w)+\tau \mathfrak{D}^\auxil\Psi(w)+\frac{\lambda}{2}\W^2(w,\tilde w)&\le \auxil(\tilde w).
\end{align*}
There, $\mathfrak{D}^\auxil\Psi(w)$ denotes the \emph{dissipation} of the functional $\Phi$ along the $\lambda$-flow $\flowS^{\auxil}$ of the functional $\auxil$, i.e.
\begin{align}
\mathfrak{D}^\auxil\Psi(w):=\limsup_{s\searrow 0}\frac{\Psi(w)-\Psi(\flowS_{s}^{\auxil}(w))}{s}.\label{eq:flowint}
\end{align}
\end{thrm}
The following theorem provides an extension of the Aubin-Lions compactness lemma to our metric setting:
\begin{thrm}[Extension of the Aubin-Lions lemma {\cite[Thm. 2]{rossi2003}}]\label{thm:ex_aub}
Let $\ban$ be a Banach space and $\mathcal{A}:\,\ban\to[0,\infty]$ be lower semicontinuous and have relatively compact sublevels in $\ban$. Let furthermore $\dg:\,\ban\times \ban\to[0,\infty]$ be lower semicontinuous and such that $\dg(u,\tilde u)=0$ for $u,\tilde u\in \operatorname{Dom}(\mathcal{A})$ implies $u=\tilde u$.

Let $(U_k)_{k\in\N}$ be a sequence of measurable functions $U_k:\,(0,T)\to \ban$. If
\begin{align}
\sup_{k\in\N}\int_0^T\mathcal{A}(U_k(t))\dd t&<\infty,\label{eq:hypo1}\\
\lim_{h\searrow 0}\sup_{k\in\N}\int_0^{T-h}\dg(U_k(t+h),U_k(t))\dd t&=0,\label{eq:hypo2}
\end{align}
then there exists a subsequence that converges in measure w.r.t. $t\in(0,T)$ to a limit $U:\,(0,T)\to \ban$.
\end{thrm}


\section{The minimizing movement scheme}\label{sec:minmov}
This section is devoted to the study of the energy functional $\Phi$ defined in \eqref{eq:energy} and of the associated minimizing movement scheme \eqref{eq:minmov}. Throughout, we assume that $u_0\in \Prob\cap H^1(\Omega)$ and that $F$ satisfies the conditions from Assumption \ref{ass}.
\begin{prop}[Properties of the energy]\label{prop:energyprop}
The energy functional $\Phi$ is proper, nonnegative and lower semicontinuous in $(\Prob,\W)$. Moreover, there exist $C_0>0$ and $C_1>0$ such that
\begin{align}
\Phi(u)&\ge C_0 \|u\|_{H^1}^2-C_1\qquad\forall u\in \Prob\cap H^1(\Omega).\label{eq:coerc}
\end{align}
\end{prop}
\begin{proof}
Obviously, Assumption \ref{ass}\eqref{ass3} yields that $0\le \Phi(u)<\infty$ for all $u\in \Prob\cap H^1(\Omega)$. For those $u$, Assumption \ref{ass}\eqref{ass3} gives $c\|\nabla u\|_{L^2}^2\le \Phi(u)$. Using $\|u\|_{L^1}=1$ and Poincaré's inequality
\begin{align*}
\left\|u-\mathcal{L}^d(\Omega)^{-1}\right\|_{L^2}^2\le K\|\nabla u_n\|_{L^2}^2,
\end{align*}
we deduce \eqref{eq:coerc}. It remains to prove that $\Phi$ is lower semicontinuous. Let $\W(\mu_n,\mu)\to 0$ as $n\to\infty$ for a sequence $(\mu_n)_{n\in\N}$ and a limit $\mu$ in $\Prob$. If $\liminf\limits_{n\to\infty} \Phi(\mu_n)=+\infty$, there is nothing to prove. We thus can assume without loss of generality that $(\Phi(\mu_n))_{n\in\N}$ is bounded. Hence $\mu_n$ is absolutely continuous for every $n\in \N$, and by Alaoglu's theorem, there exists a subsequence (non-relabelled) on which $u_n\rightharpoonup u$ in $H^1(\Omega)$. Since $F$ is jointly convex in $(x,z,p)$ by Assumption \ref{ass}\eqref{ass2} and bounded w.r.t. $x\in\Omega$, $\Phi$ is lower semicontinuous with respect to weak convergence in $H^1(\Omega)$ \cite[Thm. 10.16]{renardy2004}, completing the proof.
\end{proof}
The properties from Proposition \ref{prop:energyprop} allow us to show the well-posedness of the minimizing movement scheme \eqref{eq:minmov}. For the sake of presentation, we introduce the notation
\begin{align*}
\Phi_\tau(\cdot;v):=\frac1{2\tau}\W(\cdot,v)^2+\Phi
\end{align*}
for the Yosida penalization associated to $\Phi$, for given $\tau>0$ and $v\in\Prob$.
\begin{prop}[Minimizing movement scheme]\label{prop:mmov}
If $\tau>0$ and $v\in\Prob$ with $\Phi(v)<\infty$, then there exists a minimizer $u\in \Prob$ of $\Phi_\tau(\cdot;v)$ satisfying $\Phi(u)<\infty$.
\end{prop}
\begin{proof}
Since also $\Phi_\tau(\cdot;v)$ is nonnegative, infimizing sequences for this functional are bounded. The existence of a minimizer is now directly obtained from the properties in Proposition \ref{prop:energyprop}; recall that $u_n\rightharpoonup u$ narrowly implies in particular that $\W(u_n,v)\to \W(u,v)$, since $\Omega\subset\R^d$ is bounded.
\end{proof}
From the scheme \eqref{eq:minmov}, the following classical estimates directly follow (see for instance \cite{savare2008}):
\begin{prop}[Classical estimates]
\label{classicalestimates}
Let $\tau>0$, $u_0\in\Prob\cap H^1(\Omega)$, and define $(\utn)_{n \in \N}$ and $\ut$ by \eqref{eq:minmov}\&\eqref{eq:discsol}.
Then the following holds:
\begin{align}
\Phi(u_\tau^n)\le \Phi(u_\tau^m) &\le \Phi(u_0)\qquad \forall n\ge m\ge 0 \label{supE},\\
\sum_{n\ge 1}\W(\utnn,\utn)^2 &\le 2\tau\Phi(u_0) \label{sumw2},\\
\W(\ut(t),\ut(s)) &\le  \sqrt{2\Phi(u_0)(|t-s|+\tau)}\label{w2cont}\qquad\forall s,t\ge 0.
\end{align}
\end{prop}
Later, we need the following classical result on boundary values:
\begin{lem}[Boundary terms {\cite[Lemma 5.2]{gianazza2009}}]\label{lem:boundary}
Let $v\in C^2(\overline{\Omega})$ with $\partial_\nu v=0$ on $\partial\Omega$. Then, one has $\nabla v\cdot\nabla^2 v\nu\le 0$ on $\partial\Omega$.
\end{lem}
%


\section{Additional regularity of the discrete solution}\label{sec:addreg}
The main result of this section is concerned with the additional regularity of the minimizers in \eqref{eq:minmov} and reads
\begin{prop}[Additional regularity]\label{prop:addreg}
Let $\tau>0$ and $v\in \Prob\cap H^1(\Omega)$. Then, a minimizer $u\in \Prob\cap H^1(\Omega)$ of $\Phi_\tau(\cdot;v)$ admits the estimate
\begin{align}
\label{eq:addreg}
\int_\Omega \|\nabla^2 u\|^2\dd x&\le \frac1{\gamma\tau}\int_\Omega(v\log(v)-u\log(u))\dd x+C_3(\|u\|_{H^1}^2+1)
\end{align}
with $\gamma>0$ from Assumption \ref{ass}\eqref{ass2} and some constant $C_3\ge 0$. In particular, $u\in H^2(\Omega)$.
\end{prop}

\begin{proof}
The main idea of proof is to apply the flow interchange lemma (Theorem \ref{thm:flowinterchange}) to the free energy $\Phi$, with the following choice of the auxiliary functional: It is well-known (see for instance \cite{mccann1997}) that \emph{Boltzmann's entropy} $\ent(v)=\int_\Omega v\log(v)\dd x$ is a lower semicontinuous displacement $0$-convex functional on $(\Prob,\W)$, and its corresponding gradient flow semigroup $\flowS^\ent$ (see Definition \ref{thm:0flow}) is a solution to the heat equation with Neumann boundary conditions
\begin{xalignat}{2}
\partial_s \flowS_s^\ent(v)&=\Delta  \flowS_s^\ent(v) &&\text{on }(0,\infty)\times\Omega,\label{eq:heat}\\
\partial_\nu \flowS_s^\ent(v)&=0 &&\text{on }(0,\infty)\times\partial\Omega.\label{eq:heatbc}
\end{xalignat}
For the dissipation of $\Phi$ at $u$ along the smooth flow $\flowS^\ent$, we obtain (write $u_s:=\flowS_s^\ent(u)$ for brevity)
\begin{align*}
-\frac{\dd}{\dd s}\Phi(u_s)&=-\int_\Omega \dff_{(x,z,p)}F(x,u_s,\nabla u_s)[(0,\Delta u_s,\Delta\nabla u_s)]\dd x\\&=-\sum_{i=1}^d\int_\Omega \dff_{(x,z,p)}F(x, u_s,\nabla u_s)[\partial_{x_i}(0,\partial_{x_i} u_s,\partial_{x_i}\nabla u_s)]\dd x.
\end{align*}
Integration by parts yields with the canonical unit vectors $\ce_1,\ldots,\ce_d\in\R^d$:
\begin{align*}
-\frac{\dd}{\dd s}\Phi( u_s)&=\sum_{i=1}^d\int_\Omega \dff^2_{(x,z,p)}F(x, u_s,\nabla u_s)[(\ce_i,\partial_{x_i} u_s,\partial_{x_i}\nabla u_s),(0,\partial_{x_i} u_s,\partial_{x_i}\nabla u_s)]\dd x\\
&-\sum_{i=1}^d\int_{\partial\Omega}\dff_{(x,z,p)}F(x, u_s,\nabla u_s)[(0,\partial_{x_i} u_s,\partial_{x_i}\nabla u_s)]\nu_i\dS.
\end{align*}
For the boundary term, we use the boundary condition \eqref{eq:heatbc} and Assumption \ref{ass}\eqref{ass1} to get
\begin{align*}
&-\sum_{i=1}^d\int_{\partial\Omega}\dff_{(x,z,p)}F(x, u_s,\nabla u_s)[(0,\partial_{x_i} u_s,\partial_{x_i}\nabla u_s)]\nu_i\dS\\
&=-\int_{\partial\Omega} F_z(x, u_s,\nabla u_s)\nabla u_s\cdot\nu \dS-\int_{\partial\Omega} F_p(x, u_s,\nabla u_s)\cdot\nabla^2 u_s\nu\dS\\
&=-\int_{\partial\Omega}\tilde F_r(x, u_s,|\nabla u_s|)\frac{\nabla u_s}{|\nabla u_s|}\cdot\nabla^2 u_s\nu\dS.
\end{align*}
The last expression above is nonnegative since $\tilde F_r\ge 0$ thanks to Assumption \ref{ass}\eqref{ass1} and due to Lemma \ref{lem:boundary}. By elementary linear algebra, we have that
\begin{align*}
&\sum_{i=1}^d\int_\Omega \dff^2_{(x,z,p)}F(x, u_s,\nabla u_s)[(\ce_i,\partial_{x_i} u_s,\partial_{x_i}\nabla u_s),(0,\partial_{x_i} u_s,\partial_{x_i}\nabla u_s)]\dd x\\
&=\sum_{i=1}^d\int_\Omega \dff^2_{(x,z,p)}F(x, u_s,\nabla u_s)[(\tfrac12\ce_i,\partial_{x_i} u_s,\partial_{x_i}\nabla u_s),(\tfrac12\ce_i,\partial_{x_i} u_s,\partial_{x_i}\nabla u_s)]\dd x\\
&\quad-\sum_{i=1}^d\int_\Omega \dff^2_{(x,z,p)}F(x, u_s,\nabla u_s)[(\tfrac12\ce_i,0,0),(\tfrac12\ce_i,0,0)]\dd x\\
&\ge \gamma \int_\Omega\sum_{i=1}^d|\partial_{x_i}\nabla u_s|^2\dd x-\frac14\|\trc F_{xx}(x, u_s,\nabla u_s)\|_{L^1},
\end{align*}
using Assumption \ref{ass}\eqref{ass2} in the last step. With Assumption \ref{ass}\eqref{ass3}, one moreover has
\begin{align*}
|\trc F_{xx}(x, u_s,\nabla u_s)|&\le D(|\nabla u_s|^2+1),
\end{align*}
so all in all there exists $\tilde C>0$ such that
\begin{align*}
-\frac{\dd}{\dd s}\Phi( u_s)&\ge \gamma\int_\Omega \|\nabla^2  u_s\|^2\dd x-\tilde C(\| u_s\|_{H^1}^2+1)\ge \gamma\int_\Omega \|\nabla^2  u_s\|^2\dd x-\tilde C(\|u\|_{H^1}^2+1),
\end{align*}
where the last estimate holds since $\| u_s\|_{H^1}$ is nonincreasing w.r.t. $s>0$ due to the properties of the heat flow and Lemma \ref{lem:boundary}. Passing to the limit inferior as $s\searrow 0$ yields by lower semicontinuity that
\begin{align}
\mathfrak{D}^\ent\Phi(u)\ge \gamma\int_\Omega \|\nabla^2 u\|^2\dd x-\tilde C(\|u\|_{H^1}^2+1).\label{eq:dispheat}
\end{align}
Combining \eqref{eq:dispheat} with the flow interchange estimate \eqref{eq:flowint} obviously yields \eqref{eq:addreg}.
\end{proof}


\section{The discrete weak formulation}\label{sec:discreteweak}
We are now in position to derive a discrete version of the weak formulation \eqref{weakform}:
\begin{lem}[Discrete weak formulation]\label{lem:dweak}
Let $\tau>0$ and define the discrete solution $u_\tau$ by \eqref{eq:minmov}\&\eqref{eq:discsol}. Then, for all $\phi\in C^\infty(\overline{\Omega})$ with $\partial_\nu\phi=0$ on $\partial\Omega$ and all $\eta\in C^\infty_c((0,\infty))\cap C(\R_+)$, the following \emph{discrete weak formulation} holds:
\begin{align}
\begin{split}
&\left|\int_0^\infty\int_\Omega u_\tau(t,x)\phi(x)\frac{\eta_\tau(t+\tau)-\eta_\tau(t)}{\tau}\dd x\dd t-\int_0^\infty\int_\Omega \Non(x,u_\tau(t,x),\phi(x))\eta_\tau(t) \dd x\dd t\right|\\&\qquad\le \tau\|\phi\|_{C^2}\|\eta\|_{C^0}\Phi(u_0),
\end{split}
\label{eq:dweak}
\end{align}
where $\mathcal{N}(x,\cdot,\cdot)$ is defined as in Theorem \ref{thm:existF} and $\eta_\tau(s):=\eta\left(\left\lceil\frac{s}{\tau}\right\rceil\tau\right)$ for $s\ge 0$.
\end{lem}
\begin{proof}
We use the so-called \emph{JKO method} \cite{jko1998} (see also \cite{matthes2009,zinsl2012}) introducing a suitable perturbation of the successive minimizers in the scheme \eqref{eq:minmov}. Specifically, for fixed $\phi\in C^\infty(\overline{\Omega})$ with $\partial_\nu\phi=0$ on $\partial\Omega$, we denote by $X_{(\cdot)}:\R_+\times\Omega\to\Omega$ the smooth flow associated to the ordinary differential equation
\begin{align*}
\dot y(s)&=\nabla \phi(y(s)),\quad\text{for }s>0.
\end{align*}
Using the minimality property of $u_\tau^n$, one has for all $s>0$:
\begin{align}\label{eq:miniu}
\begin{split}
0&\le\frac1{s}\left( \Phi_\tau(X_s{_\#}u_\tau^n;u_\tau^{n-1})-\Phi_\tau(u_\tau^n;u_\tau^{n-1})\right)\\&=\frac1{s}\left(\Phi(X_s{_\#}u_\tau^n)-\Phi(u_\tau^{n})\right)+\frac1{2\tau s}\left(\W^2(X_s{_\#}u_\tau^n,u_\tau^{n-1})-\W^2(u_\tau^n,u_\tau^{n-1})\right).
\end{split}
\end{align}
The discrete weak formulation \eqref{eq:dweak} is obtained by passing to the limit as $s\searrow 0$ in \eqref{eq:miniu}. Since $u_\tau^n$ and $u_\tau^{n-1}$ are probability densities, there exists an optimal transport map $T:\R^d\to\R^d$ such that $u_\tau^{n-1}=T_\# u_\tau^n$ (see \cite{villani2003}), so consequently
\begin{align*}
\W^2(X_s{_\#}u_\tau^n,u_\tau^{n-1})-\W^2(u_\tau^n,u_\tau^{n-1})&\le \int_\Omega \left(|X_s-T|^2-|\id -T|^2\right)u_\tau^n\dd x\\&\le \int_\Omega (X_s(x)-x)\cdot(X_s(x)+x-2T(x))u_\tau^n\dd x,
\end{align*} 
using the elementary identity $|a|^2-|b|^2=(a-b)\cdot(a+b)$ for $a,b\in\R^d$ in the last step. By dominated convergence and the definition of $X_s$, one has
\begin{align*}
\lim_{s\searrow 0}\frac1{2\tau s}\left(\W^2(X_s{_\#}u_\tau^n,u_\tau^{n-1})-\W^2(u_\tau^n,u_\tau^{n-1})\right)=-\frac1{\tau}\int_\Omega \nabla\phi(x)\cdot(T(x)-x)u_\tau^n \dd x.
\end{align*}
Using the Taylor expansion $\phi(T(x))-\phi(x)=\nabla \phi(x)\cdot(T(x)-x)+\frac12 (T(x)-x)\cdot\nabla^2\phi(\tilde x)(T(x)-x)$ for some intermediate value $\tilde x$, the right-hand side above can be recast to 
\begin{align*}
-\frac1{\tau}\int_\Omega \nabla\phi(x)\cdot(T(x)-x)u_\tau^n \dd x\le\frac1{\tau}\int_\Omega (u_\tau^n-u_\tau^{n-1})\phi\dd x+\frac1{2\tau}\|\phi\|_{C^2}\W^2(u_\tau^n,u_\tau^{n-1}).
\end{align*}

To pass to the limit in the first term of the r.h.s. in \eqref{eq:miniu}, we introduce the \emph{volume distortion} $V_s:=\det(\nabla X_s)>0$ associated to the vector field $X_s$, for small $s>0$. The following identities hold (see \cite[§2]{matthes2009} for a proof):
\begin{align}\label{eq:Vsprops}
\begin{split}
\frac{\dd}{\dd s}\bigg.\bigg|_{s=0}V_s&=\Delta\phi,\qquad
\frac{\dd}{\dd s}\bigg.\bigg|_{s=0}\left[\nabla\left(\frac{u}{V_s}\circ (X_s)^{-1}\right)\circ X_s\right]=-\nabla u\Delta\phi-u\nabla\Delta\phi-\nabla^2\phi\nabla u.
\end{split}
\end{align}
Using the change of variables $x=X_s(y)$, we can rewrite $\Phi(X_s{_\#}u_\tau^n)$ as follows:
\begin{align*}
\Phi(X_s{_\#}u_\tau^n)&=\int_\Omega F\left(X_s,\frac{u_\tau^n}{V_s},\nabla\left(\frac{u_\tau^n}{V_s}\circ (X_s)^{-1}\right)\circ X_s\right)V_s\dd y.
\end{align*}
By dominated convergence, one consequently has
\begin{align*}
&\lim_{s\searrow 0}\frac1{s}\left(\Phi(X_s{_\#}u_\tau^n)-\Phi(u_\tau^{n})\right)\\&\quad=\int_\Omega \left(\dff_{(x,z,p)}F(x,u_\tau^n,\nabla u_\tau^n)\left[\frac{\dd}{\dd s}\bigg.\bigg|_{s=0}\left(X_s,\frac{u_\tau^n}{V_s},\nabla\left(\frac{u_\tau^n}{V_s}\circ (X_s)^{-1}\right)\circ X_s\right)\right]+F(x,u_\tau^n,\nabla u_\tau^n)\frac{\dd}{\dd s}\bigg.\bigg|_{s=0}V_s\right)\dd y
\\&\quad=\int_\Omega \mathcal{N}(x,u_\tau^n,\phi)\dd x,
\end{align*}
by applying \eqref{eq:Vsprops}.

Putting together, we have proved that
\begin{align*}
0\le \int_\Omega \mathcal{N}(x,u_\tau^n,\phi)\dd x+\frac1{\tau}\int_\Omega (u_\tau^n-u_\tau^{n-1})\phi\dd x+\frac1{2\tau}\|\phi\|_{C^2}\W^2(u_\tau^n,u_\tau^{n-1}).
\end{align*}
Since $\mathcal{N}(x,u_\tau^n,\phi)$ is linear with respect to $\phi$, repeating the calculations above for $-\phi$ in place of $\phi$ yields the converse inequality and hence
\begin{align}\label{eq:elg}
-\frac1{2\tau}\|\phi\|_{C^2}\W^2(u_\tau^n,u_\tau^{n-1})\le\frac1{\tau}\int_\Omega (u_\tau^n-u_\tau^{n-1})\phi\dd x& +\int_\Omega \mathcal{N}(x,u_\tau^n,\phi)\dd x\le\frac1{2\tau}\|\phi\|_{C^2}\W^2(u_\tau^n,u_\tau^{n-1}).
\end{align}
We now introduce a nonnegative temporal test function $\eta\in C^\infty_c((0,\infty))\cap C(\R_+)$, multiply \eqref{eq:elg} with $\tau\eta(n\tau)$ and sum over $n\in\N$ to obtain
\begin{align}\label{eq:elg2}
\begin{split}
-\tau\|\phi\|_{C^2}\|\eta\|_{C^0}\Phi(u_0)&\le\sum_{n\ge 1}\int_\Omega \eta(n\tau)(u_\tau^n-u_\tau^{n-1})\phi\dd x+\sum_{n\ge 1}\int_\Omega \tau\eta(n\tau)\mathcal{N}(x,u_\tau^n,\phi)\dd x\\&\qquad\le \tau\|\phi\|_{C^2}\|\eta\|_{C^0}\Phi(u_0),
\end{split}
\end{align}
using the total square distance estimate \eqref{sumw2} to simplify the last term. Rearranging above as
\begin{align*}
\sum_{n\ge 1}\int_\Omega \eta(n\tau)(u_\tau^n-u_\tau^{n-1})\phi\dd x&=\tau\sum_{n\ge 0}\int_\Omega \frac{\eta(n\tau)-\eta((n+1)\tau)}{\tau} u_\tau^n \phi \dd x,
\end{align*}
and introducing $\eta_\tau(s):=\eta\left(\left\lceil\frac{s}{\tau}\right\rceil\tau\right)$, the discrete weak formulation \eqref{eq:dweak} is obtained by rewriting \eqref{eq:elg2} in spatio-temporal integral form, recalling the definition of the discrete solution $u_\tau$. For sign-changing test functions $\eta$, we decompose into positive and negative part and subtract the respective estimates \eqref{eq:elg2} also to arrive at \eqref{eq:dweak}.
\end{proof}


\section{Passage to continuous time}\label{sec:limit}
In order to pass to the continuous-time limit $\tau\searrow 0$, the following \emph{a priori} estimates are required:
\begin{prop}[A priori estimates]\label{prop:apriori}
For each fixed $T>0$, there exist constants $C_1,C_2>0$ such that for all $\tau\in (0,T)$, the discrete solution $u_\tau$ defined via \eqref{eq:minmov}\&\eqref{eq:discsol} admits the estimates
\begin{align*}
\|u_\tau\|_{L^\infty([0,T];H^1)}&\le C_1\quad\text{and}\quad \|u_\tau\|_{L^2([0,T];H^2)}\le C_2.
\end{align*}
\end{prop}
\begin{proof}
Since $\Phi(u_0)$ is finite, the first estimate is an immediate consequence of the energy estimate \eqref{supE} from Proposition \ref{classicalestimates} and \eqref{eq:coerc} from Proposition \ref{prop:energyprop}. For the second estimate, we define $N:=\left\lceil\frac{T}{\tau}\right\rceil\in\N$ and use the additional regularity estimate \eqref{eq:addreg} from Proposition \ref{prop:addreg}:
\begin{align*}
&\int_0^T \|u_\tau\|_{H^2}^2\dd t=\int_0^T \|u_\tau\|_{H^1}^2 \dd t+\int_0^T \|\nabla ^2 u_\tau\|_{L^2}^2 \dd t\\
&\quad\le T\|u_\tau\|_{L^\infty([0,T];H^1)}^2+\sum_{n=1}^N\tau\int_\Omega\|\nabla^2 u_\tau^n\|^2\dd x\\
&\quad\le T\|u_\tau\|_{L^\infty([0,T];H^1)}^2+\sum_{n=1}^N\left[\frac1{\gamma}(\ent(u_\tau^{n-1})-\ent(u_\tau^n))+C_3\tau(\|u_\tau^n\|_{H^1}^2+1)\right].
\end{align*}
Using the elementary estimate $-\frac1{e}\le z\log(z)\le z^2$ for $z\ge 0$ in combination with the uniform estimate on $u_\tau$ in $L^\infty([0,T];H^1(\Omega))$, one obtains (for some constant $\tilde C>0$)
\begin{align*}
\|u_\tau\|_{L^2([0,T];H^2)}\le \tilde C+\frac1{\gamma}\|u_0\|_{L^2}^2+\frac{\mathcal{L}^d(\Omega)}{\gamma e},
\end{align*}
which is a finite value, thanks to the assumptions on $\Omega$ and $u_0$.
\end{proof}
We are now able to identify a candidate for the limit curve:
\begin{prop}[Convergence]\label{prop:cvg}
Let $T>0$, a vanishing sequence of step sizes $\tau_k\searrow 0$ $(k\to\infty)$ and the associated family of discrete solutions $(u_{\tau_k})_{k\in\N}$ be given. There exists a (non-relabelled) subsequence and a limit map 
$u\in C^{1/2}([0,T];(\Prob,\W))\cap L^\infty([0,T];H^1(\Omega)) \cap L^2([0,T];H^2(\Omega))$ such that for $k\to\infty$:
\begin{enumerate}[(a)]
\item $u_{\tau_k}\to u$ uniformly w.r.t. $t\in [0,T]$ in $(\Prob,\W)$.\label{a}
\item $u_{\tau_k}\to u$ in $L^2([0,T];H^1(\Omega))$.\label{b}
\item $u_{\tau_k}(t,\cdot)\to u(t,\cdot)$ in $H^1(\Omega)$ for almost every $t\in [0,T]$, as well as $u_{\tau_k}\to u$ and $\nabla u_{\tau_k}\to \nabla u$ almost everywhere in $[0,T]\times\Omega$.\label{c}
\item $u_{\tau_k}\rightharpoonup u$ in $L^2([0,T];H^2(\Omega))$.\label{d}
\end{enumerate}
\end{prop}
\begin{proof}
Part \eqref{a} is an immediate consequence of the H\"older-type estimate \eqref{w2cont} from Proposition \ref{classicalestimates} and a refined version of the Arzelà-Ascoli theorem \cite[Prop. 3.3.1]{savare2008}. The weak convergence \eqref{d} is obtained from Alaoglu's theorem and Proposition \ref{prop:apriori}. It remains to prove the strong convergence \eqref{b} which yields \eqref{c} by extraction of suitable subsequences. For the proof of \eqref{b}, we apply Theorem \ref{thm:ex_aub} to the family $(u_{\tau_k})_{k\in\N}$. Let $\ban:=H^1(\Omega)$ and define $\mathcal{A}:\ban\to[0,\infty]$ with 
\begin{align*}
\mathcal{A}(\rho):=\begin{cases}\|\rho\|_{H^2}^2,&\text{if }\rho\in H^2(\Omega),\\ +\infty,&\text{otherwise,}\end{cases}
\end{align*}
which has relatively compact sublevels in $\ban$ by the Rellich-Kondrachov compactness theorem.
Furthermore, define $\dg:\ban\times\ban\to[0,\infty]$ by
\begin{align*}
\dg(\rho,\tilde\rho):=\begin{cases}\W(\rho,\tilde\rho),&\text{if }\rho,\tilde\rho \text{ are probability densities on }\Omega,\\ +\infty,&\text{otherwise,}\end{cases}
\end{align*}
which is an admissible choice for the application of Theorem \ref{thm:ex_aub}.
Using the technique from \cite{zinsl2012}, one verifies that hypothesis \eqref{eq:hypo2} holds, thanks to the estimates from Proposition \ref{classicalestimates}. Since hypothesis \eqref{eq:hypo1} coincides with the second part of Proposition \ref{prop:apriori}, we conclude with Theorem \ref{thm:ex_aub} that (on a subsequence) $u_{\tau_k}(t,\cdot)$ converges to $u(t,\cdot)$ in $H^1(\Omega)$, in measure with respect to $t\in(0,T)$. Using the uniform estimate in $L^\infty([0,T];H^1(\Omega))$ from Proposition \ref{prop:apriori}, one has $\|u_{\tau_k}\|_{L^2([0,T];H^1)}\rightarrow \|u\|_{L^2([0,T];H^1)}$ as $k\to\infty$ by dominated convergence. With weak convergence and the Radon-Riesz theorem, we consequently have \eqref{b}.
\end{proof}
To complete the proof of Theorem \ref{thm:existF}, it remains to verify that $u$ is a weak solution to \eqref{PDE00}.
\begin{prop}[Continuous-time limit]\label{eq:disccontweak}
Let $(\tau_k)_{k\in\N}$ be the subsequence on which the convergence properties from Proposition \ref{prop:cvg} hold. Then,
\begin{align}
\lim_{k\to\infty}&\int_0^\infty\int_\Omega \left(u_{\tau_k}(t,x)\frac{\eta_{\tau_k}(t+\tau_k)-\eta_{\tau_k}}{\tau_k}-u(t,x)\partial_t\eta(t)\right)\phi(x)\dd x\dd t=0,\label{eq:convdt}\\
\lim_{k\to\infty}&\int_0^\infty\int_\Omega \bigg(\mathcal{N}(x,u_{\tau_k}(t,x),\phi(x))\eta_{\tau_k}(t)-\mathcal{N}(x,u(t,x),\phi(x))\eta(t)\bigg)\dd x\dd t=0.\label{eq:convN}
\end{align}
Moreover, for almost every $t\in [0,T]$, one has
\begin{align}
\lim_{k\to\infty}&\Phi(u_{\tau_k}(t,\cdot))=\Phi(u(t,\cdot)).\label{eq:conven}
\end{align}
\end{prop}
\begin{proof}
Since $\frac1{\tau_k}({\eta_{\tau_k}(\cdot+\tau_k)-\eta_{\tau_k}})$ converges to $\partial_t\eta$ uniformly, \eqref{eq:convdt} is an immediate consequence of the convergence properties stated in Proposition \ref{prop:cvg}. For the proof of \eqref{eq:convN}, we first note that
\begin{align*}
|\mathcal{N}(x,\rho,\phi)|\le \tilde C\|\phi\|_{C^2}(1+|\nabla\rho|^2)\qquad\forall x\in\Omega,
\end{align*}
for some constant $\tilde C>0$, thanks to the bounds on $F$ from Assumption \ref{ass}\eqref{ass2}\&\eqref{ass3}. By Proposition \ref{prop:cvg}, we have $\mathcal{N}(x,u_{\tau_k}(t,x),\phi(x))\eta_{\tau_k}(t)\rightarrow\mathcal{N}(x,u(t,x),\phi(x))\eta(t)$ pointwise a.e. on $[0,T]\times\Omega$. Convergence of the integral is obtained by Vitali's convergence theorem as follows: Let $\eps>0$. Since $u_{\tau_k}\to u$ in $L^2([0,T];H^1(\Omega))$, there exists $\delta\in \left(0,\frac{\eps}{2\tilde C\|\phi\|_{C^2}\|\eta\|_{C^0}}\right)$ such that for all Borel sets $A\subset [0,T]\times\Omega$ with $\mathcal{L}^{d+1}(A)\le \delta$ and all $k\in\N$, one has
\begin{align*}
\iint_A |\nabla u_{\tau_k}|^2\dd x\dd t&\le \frac{\eps}{2\tilde C\|\phi\|_{C^2}\|\eta\|_{C^0}}.
\end{align*}
Consequently, for those $A$ and all $k\in\N$,
\begin{align*}
\iint_A |\mathcal{N}(x,u_{\tau_k}(t,x),\phi(x))\eta_{\tau_k}(t)|\dd x\dd t&\le \tilde C\|\phi\|_{C^2}\|\eta\|_{C^0}\left(\delta+\frac{\eps}{2\tilde C\|\phi\|_{C^2}\|\eta\|_{C^0}}\right)\le \eps,
\end{align*}
which allows us to apply Vitali's convergence theorem to the family $(\mathcal{N}(\cdot,u_{\tau_k},\phi)\eta_{\tau_k})_{k\in\N}$. The proof of \eqref{eq:conven} can be obtained by a similar argument: Starting from the first part of \eqref{c} from Proposition \ref{prop:cvg}, we deduce that for every subsequence of $(\tau_k)_{k\in\N}$, there exists a sub-subsequence $(\tau_{\tilde k})_{\tilde k\in\N}$ on which $\Phi(u_{\tau_{\tilde k}})$ converges to $\Phi(u)$ as $\tilde k\to\infty$. For the convergence, we employ Vitali's theorem in almost the same way as above; the necessary uniform integrability is a consequence of Assumption \ref{ass}\eqref{ass2}.
\end{proof}
Obviously, \eqref{eq:convdt}\&\eqref{eq:convN} imply that $u$ satisfies the weak formulation \eqref{weakform} in Theorem \ref{thm:existF}. The remaining assertion on the monotonicity of $\Phi(u(t,\cdot))$ w.r.t. $t$ is an immediate consequence of the discrete estimate \eqref{supE} and the convergence \eqref{eq:conven}.

\section{Extension to non-convex Lagrangians}\label{sec:f}
In this section, we sketch the possible extension of the strategy described above to the case covered by Assumption \ref{assf}. For the sake of brevity, we skip most of the technical details. Recall that the free energy is defined as
\begin{align*}
\Phi(u)=\int_\Omega \frac12|\nabla f(u)|^2\dd x,
\end{align*}
if $u\in\Prob$ and $f(u)\in H^1(\Omega)$; and $\Phi(u)=+\infty$ otherwise.

We first summarize several properties on the nonlinearity $f$ which are elementary consequences of Assumption \ref{assf}:
\begin{lem}[Properties of $f$]\label{lem:propf}
The following statements hold:
\begin{enumerate}[(a)]
\item With $C$ and $\alpha$ from Assumption \ref{assf}, one has $f(z)\ge\frac{C}{\alpha}z^\alpha$ for all $z\ge 0$.
\item There exists a constant $C_0>0$ such that $f(z)\le C_0(z+1)$.
\item There exists a constant $C_1>0$ such that $zf'(z)\le C_1(f(z)+1)$ for all $z>0$.
\end{enumerate}
\end{lem}
We are now in position to construct a discrete solution via the minimizing movement scheme \eqref{eq:minmov}:
\begin{prop}[Minimizing movement scheme]\label{prop:mmovf}
If $\tau>0$ and $v\in\Prob$ with $f(v)\in H^1(\Omega)$, then there exists a minimizer $u\in \Prob$ of $\Phi_\tau(\cdot;v)$ satisfying $f(u)\in H^1(\Omega)$.
\end{prop}
\begin{proof}
Let an infimizing sequence $(u_k)_{k\in\N}$ for $\Phi_\tau(\cdot,v)$ be given. Since $\Phi$ is nonnegative, $\Phi_\tau(u_k;v)$ is bounded; hence
$$\W(u_k,v)\le C\quad\text{and}\quad \Phi(u_k)\le C$$
for some $C>0$. Using Poincaré's inequality together with Lemma \ref{lem:propf} yields $\|f(u_k)\|_{H^1}\le \tilde C$ since 
\begin{align*}
\mathcal{L}(\Omega)^{-1}\int_\Omega f(u_k)\dd x\le C_0 \mathcal{L}(\Omega)^{-1}\int_\Omega (u_k+1)\dd x =C_0(\mathcal{L}(\Omega)^{-1}+1).
\end{align*}
Using Alaoglu's theorem in combination with the Rellich-Kondrachov compactness theorem, one has (on a non-relabelled subsequence) that $w_k:=f(u_k)$ converges to some $w$ weakly in $H^1(\Omega)$ as well as strongly in $L^2(\Omega)$ and pointwise almost everywhere on $\Omega$. By Prokhorov's theorem, $u_k\to u$ in $(\Prob,\W)$, on a further subsequence. Clearly, $u=f^{-1}(w)$, so $f(u_k)\rightharpoonup f(u)$ in $H^1(\Omega)$. By weak lower semicontinuity of $\Phi_\tau(\cdot;v)$ (following by weak lower semicontinuity of $\frac12\|\nabla w\|^2_{L^2}$), it is easy to conclude that $u$ is a minimizer.
\end{proof}
As before, the minimizer gains in regularity:
\begin{prop}[Additional regularity]\label{prop:addregf}
Let $\tau>0$ and $v\in \Prob$ with $f(v)\in H^1(\Omega)$. Then, a minimizer $u\in \Prob$ of $\Phi_\tau(\cdot;v)$ admits the estimate
\begin{align}
\int_\Omega \|\nabla^2 f(u)\|^2\dd x&\le \frac{1}{\delta\tau}\int_\Omega(v\log (v)-u\log(u))\dd x,\label{eq:addregf}
\end{align}
where $\delta>0$ is a constant depending only on the structure of $f$ and the spatial dimension $d$.
\end{prop}
\begin{proof}
We proceed as in the proof of Proposition \ref{prop:addreg}. Recall the definition of the heat flow $\flowS^\ent$ from \eqref{eq:heat}\&\eqref{eq:heatbc}. Clearly, estimate \eqref{eq:addregf} is obtained via the flow interchange lemma (Theorem \ref{thm:flowinterchange}) if the dissipation along the heat flow admits the estimate
\begin{align}\label{eq:dsp0}
\mathfrak{D}^\ent\Phi(u)&\ge \delta\int_\Omega\|\nabla^2 f(u)\|^2\dd x.
\end{align}
Denote, for brevity, $f_s:=f(\flowS^\ent_s(u))$ (and analogously $f'_s$ etc.) and $u_s:=\flowS^\ent_s(u)$, for $s>0$. Since $u_s$ is smooth and strictly positive, one obtains
\begin{align*}
\mathfrak{D}_s:=-\frac{\dd}{\dd s}\Phi(u_s)=\int_\Omega \Delta f_s f'_s\Delta u_s\dd x=\int_\Omega (\Delta f_s)^2\dd x-\int_\Omega \Delta f_s f''_s|\nabla u_s|^2,
\end{align*}
using the product rule $\Delta f_s=f''_s |\nabla u_s|^2+f'_s \Delta u_s$. Since $\Delta f_s\nabla f_s\cdot \nu=0$ and $\nabla f_s \cdot \nabla^2 f_s\nu\le 0$ on $\partial\Omega$ (recall Lemma \ref{lem:boundary}), we have with integration by parts that
\begin{align*}
\mathfrak{D}_s-\delta \int_\Omega \|\nabla ^2 f_s\|^2\dd x \ge (1-\delta)\int_\Omega \|\nabla ^2 f_s\|^2\dd x-\int_\Omega \nabla f_s\cdot\nabla\left(H(f_s)|\nabla f_s|^2\right)\dd x,
\end{align*}
where $\delta\in (0,1)$ is a parameter to be specified below and $H(f_s):=-\frac{f''_s}{(f_s')^2}$. Note that $\nabla H(f_s)=L(f_s)\nabla f_s$ with $L(f_s)=-\frac{f_s'''f'_s+2(f''_s)^2}{(f'_s)^5}$. By Gau\ss' theorem, we first obtain
\begin{align*}
0&=\int_\Omega \nabla\cdot (H(f_s)|\nabla f_s|^2\nabla f_s)\dd x=\int_\Omega \left[H(f_s)\Delta f_s |\nabla f_s|^2+2H(f_s)\nabla f_s\cdot \nabla^2 f_s\nabla f_s+L(f_s)|\nabla f_s|^4\right]\dd x,
\end{align*}
and consequently, for $\chi:=\sqrt{\frac{d}{d+8}}$:
\begin{align}
\begin{split}
\mathfrak{D}_s-\delta \int_\Omega \|\nabla ^2 f_s\|^2\dd x \ge &(1-\delta)\int_\Omega \|\nabla ^2 f_s\|^2\dd x\\&+\int_\Omega \left[H(f_s)\left((1-\chi)\Delta f_s|\nabla f_s|^2-2\chi \nabla f_s\cdot \nabla^2 f_s\nabla f_s \right)-\chi L(f_s)|\nabla f_s|^4\right]\dd x.
\end{split}
\label{eq:dsp1}
\end{align}
Defining $R:=\nabla^2 f_s-\frac{\Delta f_s}{d}\Id$ (which is traceless and symmetric) and choosing the Frobenius norm as matrix norm, we observe that
\begin{align*}
\|\nabla^2 f_s\|^2=\trc(\nabla^2 f_s\nabla^2 f_s)=\frac{(\Delta f_s)^2}{d}+\|R\|^2.
\end{align*}
Insert this into \eqref{eq:dsp1} to find
\begin{align}\label{eq:dsp2}
\begin{split}
\mathfrak{D}_s-\delta \int_\Omega \|\nabla ^2 f_s\|^2\dd x \ge \int_\Omega\bigg[&\frac{1-\delta}{d}(\Delta f_s)^2+\left(1-\left(1+\frac2{d}\right)\chi\right)H(f_s)\Delta f_s|\nabla f_s|^2\bigg.\\&\bigg.+(1-\delta)\|R\|^2-2\chi H(f_s)\nabla f_s\cdot R\nabla f_s-\chi L(f_s)|\nabla f_s|^4\bigg]\dd x.
\end{split}
\end{align}
Applying Young's inequality to the first two terms on the r.h.s. of \eqref{eq:dsp2} and using Lemma \ref{lem:traceless} from Appendix \ref{app:A} below to estimate the third and fourth term, we end up with
\begin{align}\label{eq:dsp3}
\mathfrak{D}_s-\delta \int_\Omega \|\nabla ^2 f_s\|^2\dd x&\ge \chi\int_\Omega H(f_s)^2|\nabla f_s|^4\left[-\frac{L(f_s)}{H(f_s)^2}-\frac{1}{1-\delta}\left(-1-\frac{d}{2}+\frac12\sqrt{d^2+8d}\right)\right]\dd x.
\end{align}
We conclude the proof by showing that there exists $\delta\in (0,1)$ such that the expression in square brackets in \eqref{eq:dsp3} is nonnegative. Indeed, since
\begin{align*}
-\frac{L(f_s)}{H(f_s)^2}=\frac{f'''_sf_s'}{(f_s'')^2}-2\ge -1-\frac{d}{2}+\frac12\sqrt{d^2+8d}+\bar\delta
\end{align*}
by Assumption \ref{assf}, one has
\begin{align*}
-\frac{L(f_s)}{H(f_s)^2}-\frac{1}{1-\delta}\left(-1-\frac{d}{2}+\frac12\sqrt{d^2+8d}\right)&\ge \frac1{1-\delta}\left[\bar\delta-\delta\left(\bar\delta-1-\frac{d}{2}+\frac12\sqrt{d^2+8d}\right)\right],
\end{align*}
which is nonnegative choosing $\delta$ sufficiently small. The desired estimate \eqref{eq:dsp0} now follows by passage to $s\searrow 0$ and lower semicontinuity.
\end{proof}
\begin{rem}[Power functions]\label{rem:power}
In consistence with the results from \cite{matthes2009}, our theory comprises the important case that $f(z)=Cz^\alpha$ is a genuine concave power function, where the exponent $\alpha$ is allowed inside the range
\begin{align}\label{eq:condalpha}
\frac{3}{4}-\frac14\sqrt{1+\frac{8}{d}}<\alpha\le 1.
\end{align}
Note that, in view of Assumption \ref{assf}, one has
\begin{align*}
\frac12>\frac{3}{4}-\frac14\sqrt{1+\frac{8}{d}}\ge \frac12-\frac1{d}\quad\text{ and }\quad \frac{f'''(z)f'(z)}{f''(z)^2}=\frac{2-\alpha}{1-\alpha}\quad\forall z>0,
\end{align*}
hence Assumption \ref{assf} is equivalent to \eqref{eq:condalpha} for this special choice of $f$.
\end{rem}
With all these prerequisites at hand, we derive a discrete weak formulation and \emph{a priori} estimates satisfied by the discrete solution $u_\tau$.
\begin{prop}[Discrete weak formulation and \emph{a priori} estimates]\label{prop:dweakf}
Let $\tau>0$ and define the discrete solution $u_\tau$ by \eqref{eq:minmov}\&\eqref{eq:discsol}. Then, for all $\beta>0$, $\phi\in C^\infty(\overline{\Omega})$ with $\partial_\nu\phi=0$ on $\partial\Omega$ and all $\eta\in C^\infty_c((0,\infty))\cap C(\R_+)$, the following \emph{discrete weak formulation} holds:
\begin{align}\label{eq:dweakf}
\begin{split}
&-\|\eta\|_{C^0}\kappa\tau\Phi(u_0)+\beta\int_0^\infty\int_\Omega \frac{|\eta|_\tau(t)-|\eta|_\tau(t+\tau)}{\tau}u_\tau(t,x)\log u_\tau(t,x)\dd x\dd t\\
&\quad\le \int_0^\infty\int_\Omega \frac{\eta_\tau(t)-\eta_\tau(t+\tau)}{\tau}u_\tau\phi\dd x\dd t+\int_0^\infty\int_\Omega \eta_\tau\Non_f(u_\tau,\phi)\dd x \dd t\\
&\quad\le \|\eta\|_{C^0}\kappa\tau\Phi(u_0)-\beta\int_0^\infty\int_\Omega \frac{|\eta|_\tau(t)-|\eta|_\tau(t+\tau)}{\tau}u_\tau(t,x)\log u_\tau(t,x)\dd x\dd t,
\end{split}
\end{align}
where $\Non_f$ is defined as in Theorem \ref{thm:existf} and $\kappa\ge 0$ is such that $-\kappa \Id \le \nabla^2\phi(x) \le \kappa\Id$ for all $x\in\overline{\Omega}$. For $h:\R_+\to \R$, we denote $h_\tau(s)=h(\lceil\frac{s}{\tau}\rceil\tau)$.

Furthermore, the following \emph{a priori} estimates hold for each fixed $T>0$:
\begin{align}\label{eq:aprif}
\|f(u_\tau)\|_{L^\infty([0,T];H^1)}\le C_1,\quad \|f(u_\tau)\|_{L^2([0,T];H^2)}\le C_2\quad\text{and}\quad \|u_\tau\|_{L^\infty([0,T];L^p)}\le C_3,
\end{align}
for some constants $C_j>0$ and some $p>1$.
\end{prop}
\begin{proof}
The \emph{a priori} estimates on $u_\tau$ are obtained in the same way as in the proof of Proposition \ref{prop:apriori}. Note that, thanks to Assumption \ref{assf} and Sobolev's inequality, we have for all $t\in [0,T]$:
\begin{align}\label{eq:Lpest}
\begin{split}
\|u_\tau(t,\cdot)\|_{L^p}&\le \tilde C\qquad\text{for all }p\in \left(1,\alpha\frac{2d}{d-2}\right]\qquad\text{if }d\ge 3,\\
\|u_\tau(t,\cdot)\|_{L^p}&\le \tilde C\qquad\text{for all }p\in (1,\infty)\qquad\text{if }d=2,\\
\|u_\tau(t,\cdot)\|_{L^p}&\le \tilde C\qquad\text{for all }p\in (1,\infty]\qquad\text{if }d=1.
\end{split}
\end{align}
To derive the discrete weak formulation \eqref{eq:dweakf}, we use the flow interchange lemma (Theorem \ref{thm:flowinterchange}) with the displacement $(-\kappa)$-convex \emph{regularized potential energy}
\begin{align*}
\pot(v):=\beta \int_\Omega v\log v\dd x+\int_\Omega v\phi\dd x,
\end{align*}
for $\beta>0$. The associated $(-\kappa)$-flow $\flowS^\pot$ is given by (see \cite{savare2008} for more details)
\begin{xalignat*}{2}
\partial_s \flowS_s^\pot(v)&=\beta \Delta \flowS_s^\pot(v)+\nabla\cdot (\flowS_s^\pot(v)\nabla\phi) &&\text{on }(0,\infty)\times\Omega,\\
\partial_\nu \flowS_s^\pot(v)&=0 &&\text{on }(0,\infty)\times\partial\Omega.
\end{xalignat*}
Let $\tau>0$ and $n\in\N$. Since $\flowS_s^\pot(u_\tau^n)$ is smooth and strictly positive, we obtain for $s>0$ (writing again $u_s:=\flowS_s^\pot(u_\tau^n)$ and $f_s:=f(u_s)$ for brevity):
\begin{align*}
-\frac{\dd}{\dd s}\Phi(u_s)=\int_\Omega f'(u_s)\Delta f(u_s)(\beta \Delta u_s+\nabla \cdot (u_s\nabla \phi))\dd x.
\end{align*}
The viscosity term can be treated exactly as in the proof of Proposition \ref{prop:addregf} yielding 
\begin{align*}
-\frac{\dd}{\dd s}\Phi(u_s)=\int_\Omega \left[\beta\delta \|\nabla^2 f_s\|^2+\Delta f(u_s)\nabla f_s\cdot\nabla \phi+\Delta f_s f'(u_s)u_s\Delta\phi\right]\dd x.
\end{align*}
Using H\"older's inequality and Lemma \ref{lem:propf}(c), we can estimate from below:
\begin{align}\label{eq:dspabgesch}
-\frac{\dd}{\dd s}\Phi(u_s)&\ge \beta\delta \|\nabla^2 f_s\|_{L^2}^2-C\|\nabla^2 f_s\|_{L^2}(\|f_s\|_{H^1}+1),
\end{align}
for some constant $C>0$ depending on $\phi$. The following Sobolev inequality holds for all $\rho\in H^2(\Omega)$:
\begin{align}\label{eq:erstaunlich}
\|\rho\|_{H^1}&\le C'\|\nabla^2 \rho\|_{L^2}^\theta\|\rho\|_{L^1}^{1-\theta},\qquad\text{with }\theta=\frac{1+\frac{d}{2}}{2+\frac{d}{2}}\in(0,1).
\end{align}
Using \eqref{eq:erstaunlich} in \eqref{eq:dspabgesch}, recalling that $\|f_s\|_{L^1}\le C_0(\|u_s\|_{L^1}+\mathcal{L}^d(\Omega))=C_0(1+\mathcal{L}^d(\Omega))$ by Lemma \ref{lem:propf}(b), one obtains
\begin{align}\label{eq:dspabgesch2}
-\frac{\dd}{\dd s}\Phi(u_s)&\ge \beta\delta \|\nabla^2 f_s\|_{L^2}^2-\tilde C \|\nabla^2 f_s\|_{L^2}^{1+\theta}-\tilde C\|\nabla^2 f_s\|_{L^2}\ge \frac12\beta\delta \|\nabla^2 f_s\|_{L^2}^2-\tilde C',
\end{align}
which is bounded from below, using Young's inequality in the last step. We are now concerned with the passage to the limit inferior as $s\searrow 0$. In view of the flow interchange lemma (Theorem \ref{thm:flowinterchange}), we can assume that at least for small $s>0$ the quantity $-\frac{\dd}{\dd s}\Phi(u_s)$ is bounded from above. By \eqref{eq:dspabgesch2}, we infer that $\|\nabla^2 f_s\|_{L^2}$ is bounded in $s$. We conclude that (on a suitable subsequence) $f(u_s)\rightharpoonup f(u_\tau^n)$ in $H^2(\Omega)$ as well as $f(u_s)\to f(u_\tau^n)$ in $H^1(\Omega)$ by Rellich's theorem, as $s\searrow 0$. By continuity of $f$, one also has $u_s\to u_\tau^n$ pointwise almost everywhere on $\Omega$. By a straightforward application of Vitali's theorem and Lemma \ref{lem:propf}(c), one gets that $f'(u_s)u_s\to f'(u_\tau^n)u_\tau^n$ in $L^2(\Omega)$. Hence, weak-strong convergence leads to
\begin{align*}
\lim_{s\searrow 0}\int_\Omega\left[\Delta f(u_s)\nabla f_s\cdot\nabla \phi+\Delta f_s f'(u_s)u_s\Delta\phi\right]\dd x&=\int_\Omega\left[\Delta f(u_\tau^n)\nabla f(u_\tau^n)\cdot\nabla \phi+\Delta f(u_\tau^n) f'(u_\tau^n)u_\tau^n\Delta\phi\right]\dd x\\&=\int_\Omega \Non_f(u_\tau^n,\phi)\dd x.
\end{align*}
We thus arrive at
\begin{align*}
\mathfrak{D}^\pot(\Phi(u_\tau^n))&\ge \int_\Omega \Non_f(u_\tau^n,\phi)\dd x.
\end{align*}
The flow interchange lemma (Theorem \ref{thm:flowinterchange}) then yields for $\phi$ and $-\phi$, respectively:
\begin{align*}
&\quad -\frac{\kappa}{2}\W^2(u_\tau^n,u_\tau^{n-1})-\beta \int_\Omega u_\tau^{n-1}\log u_\tau^{n-1}\dd x+\beta \int_\Omega u_\tau^n\log u_\tau^n\dd x\\ 
&\le \int_\Omega u_\tau^n\phi\dd x -\int_\Omega u_\tau^{n-1}\phi\dd x+\tau\int_\Omega\Non_f(u_\tau^n,\phi)\dd x\\
&\quad \le \frac{\kappa}{2}\W^2(u_\tau^n,u_\tau^{n-1})+\beta \int_\Omega u_\tau^{n-1}\log u_\tau^{n-1}\dd x-\beta \int_\Omega u_\tau^n\log u_\tau^n\dd x.
\end{align*}
Proceeding similarly as for Lemma \ref{lem:dweak}, we introduce a \emph{nonnegative} test function $\eta\in C^\infty_c((0,\infty))\cap C(\R_+)$, multiply the chain of inequalites above with $\eta(n\tau)$ and sum over $n\in\N$ to eventually get
\begin{align}\label{eq:dwplus}
\begin{split}
&\quad-\|\eta\|_{C^0}\kappa\tau\Phi(u_0)+\beta\int_0^\infty\int_\Omega \frac{\eta_\tau(t)-\eta_\tau(t+\tau)}{\tau}u_\tau\log(u_\tau)\dd x\dd t\\
&\le \int_0^\infty\int_\Omega \frac{\eta_\tau(t)-\eta_\tau(t+\tau)}{\tau}u_\tau\phi\dd x\dd t+\int_0^\infty\int_\Omega \eta_\tau\Non_f(u_\tau,\phi)\dd x \dd t\\
&\quad\le \|\eta\|_{C^0}\kappa\tau\Phi(u_0)-\beta\int_0^\infty\int_\Omega \frac{\eta_\tau(t)-\eta_\tau(t+\tau)}{\tau}u_\tau\log(u_\tau)\dd x\dd t.
\end{split}
\end{align}
To arrive at \eqref{eq:dweakf} for \emph{arbitrary} test functions $\eta\in C^\infty_c((0,\infty))\cap C(\R_+)$, we decompose $\eta=\eta_+-\eta_-$ into its positive and negative parts and use \eqref{eq:dwplus} for $\eta_+$ and $\eta_-$, respectively (recall that $|\eta|=\eta_++\eta_-$).
\end{proof}

As in the previous section, the proof of Theorem \ref{thm:existF} is completed by passing to the continuous-time limit $\tau\searrow 0$.
\begin{prop}[Continuous-time limit]\label{prop:convergencef}
Let $T>0$, a vanishing sequence of step sizes $\tau_k\searrow 0$ $(k\to\infty)$ and the associated family of discrete solutions $(u_{\tau_k})_{k\in\N}$ be given. Then, the following statements hold with $p>1$ from Proposition \ref{prop:dweakf}:

There exists a (non-relabelled) subsequence and a map $u\in C^{1/2}([0,T];(\Prob,\W))\cap L^\infty([0,T];L^p(\Omega))$ such that for $k\to\infty$:
\begin{enumerate}[(a)]
\item $u_{\tau_k}\to u$ uniformly w.r.t. $t\in [0,T]$ in $(\Prob,\W)$.
\item $u_{\tau_k}\to u$ in $L^p([0,T];L^p(\Omega))$ and almost everywhere on $[0,T]\times\Omega$.
\item $f(u_{\tau_k})\to f(u)$ in $L^2([0,T];H^1(\Omega))$.
\item $f(u_{\tau_k})\rightharpoonup f(u)$ in $L^2([0,T];H^2(\Omega))$.
\end{enumerate}
The limit $u$ is a weak solution to \eqref{PDE00} in the sense stated in Theorem \ref{thm:existf} and for almost every $t\in [0,T]$, one has
\begin{align*}
\lim_{k\to\infty}&\Phi(u_{\tau_k}(t,\cdot))=\Phi(u(t,\cdot)).
\end{align*}
\end{prop}
\begin{proof}
Clearly, $f(u_{\tau_k})\rightharpoonup v$ in $L^2([0,T];H^2(\Omega))$ for some limit $v$ by Alaoglu's theorem and estimate \eqref{eq:aprif}. We show that $f(u_\tau)\to f(u)$ in $L^2([0,T];L^2(\Omega))$ for $u=f^{-1}(v)$. Then, by a standard interpolation inequality, $f(u_\tau)\to f(u)$ in $L^2([0,T];H^1(\Omega))$ follows. We seek to apply Theorem \ref{thm:ex_aub} and let $\ban:=L^p(\Omega)$ (with $p>1$ from Proposition \ref{prop:dweakf}),
\begin{align*}
\mathcal{A}(\rho):=\begin{cases}\|f(\rho)\|_{H^1}^2,&\text{if }f(\rho)\in H^1(\Omega),\\ +\infty,&\text{otherwise},\end{cases}
\end{align*}
and $\dg$ as in the proof of Proposition \ref{prop:cvg}. Obviously, $\int_0^T\mathcal{A}(u_{\tau_k}(t,\cdot))\dd t$ is $k$-uniformly bounded thanks to \eqref{eq:aprif}. For the application of Theorem \ref{thm:ex_aub}, it remains to verify that $\mathcal{A}$ has relatively compact sublevels in $\ban$. Let a sequence $(\rho_k)_{k\in\N}$ in $\ban$ with $\mathcal{A}(\rho_k)\le C$ for all $k\in\N$ and some $C\ge 0$ be given. Clearly, $\|f(\rho_k)\|_{H^1}\le \sqrt{C}$ for all $k$; hence, by the Rellich-Kondrachov compactness theorem, $f(\rho_k)\to f(\rho)$ in $L^r(\Omega)$ for sufficiently large $r\in \left(1,\frac{2d}{d-2}\right)$ as well as pointwise a.e. on $\Omega$, for some limit $f(\rho)\in L^r(\Omega)$, extracting a suitable subsequence. By continuity, $\rho_k\to \rho$ pointwise a.e. on $\Omega$. Without loss of generality, we may also assume that $\rho_k\to\rho$ weakly in $\ban$, by Alaoglu's theorem and \eqref{eq:Lpest} --- which also yields $L^p$-uniform integrability of $(\rho_k)_{k\in\N}$. Convergence of $\rho_k$ to $\rho$ in $\ban$ then follows by Vitali's convergence theorem.

The application of Theorem \ref{thm:ex_aub} yields the existence of a limit map $u$ such that (on a subsequence) $u_{\tau_k}(t,\cdot)\to u(t,\cdot)$ in $\ban$, in measure w.r.t. $t\in [0,T]$. As in the proof of Proposition \ref{prop:cvg}, we deduce that $u_{\tau_k}\to u$ in $L^p([0,T];L^p(\Omega))$ and pointwise almost everywhere on $[0,T]\times\Omega$, possibly extracting further subsequences. Hence, by continuity, $f(u_{\tau_k})\to f(u)$ a.e. on $[0,T]\times\Omega$. Since $f(u_{\tau_k})$ is $k$-uniformly bounded in $L^\infty([0,T];L^{\frac{2d}{d-2}}(\Omega))$ thanks to Sobolev's inequality and \eqref{eq:aprif}, the family $(f(u_{\tau_k}))_{k\in\N}$ is uniformly integrable in $L^2([0,T];L^2(\Omega))$. Vitali's theorem yields the claimed convergence in $L^2([0,T];L^2(\Omega))$.

The limit $u$ is a weak solution to \eqref{PDE00} because of the following: First, since $-\frac1{e}\le z\log z\le z^p$ for all $z\ge 0$ and thanks to the uniform estimate in $L^\infty([0,T];L^p(\Omega))$, $u_{\tau_k}\log(u_{\tau_k})$ is $k$-uniformly bounded in $L^1([0,T];L^1(\Omega))$. Hence, letting $k\to\infty$ and $\beta\searrow 0$ in \eqref{eq:dweakf} yields
\begin{align*}
\lim_{k\to\infty}\left[\int_0^\infty\int_\Omega \frac{\eta_{\tau_k}(t)-\eta_{\tau_k}(t+{\tau_k})}{{\tau_k}}u_{\tau_k}\phi\dd x\dd t+\int_0^\infty\int_\Omega \eta_{\tau_k}\Non_f(u_{\tau_k},\phi)\dd x \dd t\right]&=0.
\end{align*}
We obtain the time-continuous weak formulation since 
\begin{align*}
\lim_{k\to\infty}\int_0^\infty\int_\Omega \eta_{\tau_k} \Non_f(u_{\tau_k},\phi)\dd x \dd t&=
\int_0^\infty\int_\Omega \eta\Non_f(u,\phi)\dd x \dd t,
\end{align*}
thanks to the weak convergence $\Delta f(u_{\tau_k})\rightharpoonup \Delta f(u)$ and the strong convergence $\nabla f(u_{\tau_k})\to\nabla f(u)$, both in $L^2([0,T];L^2(\Omega))$. Especially, $u_{\tau_k}f'(u_{\tau_k})\to uf'(u)$ in $L^2([0,T];L^2(\Omega))$ due to Lemma \ref{lem:propf} and Vitali's theorem:
\begin{align*}
\int_0^T\int_\Omega (u_{\tau_k}f'(u_{\tau_k}))^2\dd x\dd t &\le C_1 \int_0^T\int_\Omega (f(u_{\tau_k})+1)^2\dd x \dd t\le 2C_1 \int_0^T\int_\Omega (f(u_{\tau_k})^2+1)\dd x \dd t,
\end{align*}
and $(f(u_{\tau_k}))_{k\in\N}$ is uniformly integrable in $L^2([0,T];L^2(\Omega))$.
\end{proof}
%


\appendix

\section{}\label{app:A}

\begin{lem}[Binomial formula for traceless symmetric matrices]\label{lem:traceless}
Let $A\in \R^{d\times d}$ be symmetric and traceless. Then, for all $v\in\R^d$, the following estimate holds:
\begin{align}
\|A\|^2+2v\cdot Av+\frac{d-1}{d}|v|^4&\ge 0.\label{eq:Av}
\end{align}
\end{lem}
\begin{proof}
Without loss of generality, $|v|=1$. Using the orthogonal decomposition $A=SDS^\tT$ for orthogonal $S$ and diagonal $D$, \eqref{eq:Av} is --- by the transformation $w:=S^\tT v$ (note that $|w|=1$) --- equivalent to 
\begin{align*}
\sum_{i=1}^d \lambda_i^2+2\sum_{i=1}^d \lambda_i w_i^2+\frac{d-1}{d}\ge 0,
\end{align*}
where $\lambda_1,\ldots,\lambda_d$ are the (real) eigenvalues of $A$. Without restriction, assume that $\lambda_1$ is the smallest eigenvalue of $A$. Since $A$ is traceless, we have $\sum\limits_{i=1}^d\lambda_i=0$. Since $\sum\limits_{i=1}^d \lambda_i w_i^2\ge \lambda_1$, we obtain
\begin{align*}
\sum_{i=1}^d &\lambda_i^2+2\sum_{i=1}^d \lambda_i w_i^2+\frac{d-1}{d}\ge \sum_{i=2}^d \lambda_i^2+\left(\sum_{i=2}^d\lambda_i\right)^2-2\sum_{i=2}^d \lambda_i+\frac{d-1}{d}\\
&=\sum_{i=2}^d\left(\lambda_i-\frac1{d}\right)^2+\left[\sum_{i=2}^d\left(\lambda_i-\frac1{d}\right)\right]^2\ge 0. \qedhere
\end{align*}
\end{proof}


\bibliographystyle{abbrv}
   \bibliography{bib}

\end{document}